\documentclass{article}
\usepackage{amsmath,amsthm,amscd,amssymb}
\usepackage[mathscr]{eucal}
\usepackage{amssymb}
\usepackage{latexsym}

\theoremstyle{definition}
\newtheorem{Def}{Definition}[section]
\newtheorem{Thm}[Def]{Theorem}
\newtheorem{Prop}[Def]{Proposition}
\newtheorem{Rem}[Def]{Remark}

\newtheorem{Lem}[Def]{Lemma}

\numberwithin{equation}{section}

\begin{document}

\title{On $p$-adic quaternionic Eisenstein series}

\author{Toshiyuki Kikuta and Shoyu Nagaoka}
\maketitle






\maketitle


\begin{abstract}
We show that certain $p$-adic Eisenstein series for quaternionic
modular groups of degree 2 become ``real" modular forms of level $p$ in almost all cases.
To prove this, we introduce a $U(p)$ type operator. We also show that there 
exists a $p$-adic Eisenstein series of the above type that has transcendental coefficients. 
Former examples of $p$-adic Eisenstein series for Siegel and Hermitian modular groups are both rational (i.e., algebraic).
\end{abstract}

\section{Introduction}
Serre \cite{Se} first developed the theory of $p$-adic Eisenstein series and there have subsequently been many results in the field of $p$-adic modular forms. 
Several researchers have attempted to generalize the theory to modular forms with several variables. For example, we showed that a $p$-adic
limit of a Siegel Eisenstein series becomes a ``real" Siegel 
modular form (cf. \cite{KN}). The same result has also been proved for Hermitian 
modular forms (e.g., \cite{N}).

In the present paper, we study $p$-adic limits of quaternionic Eisenstein series. This study has two principal aims. 
The first is to show that these $p$-adic limits become ``real" modular forms of level $p$ for higher 
$p$-adical weights (Theorem \ref{ThmM}). 
To prove this, we introduce a $U(p)$ type Hecke operator and study its properties; this is 
a similar method to that used by B\"ocherer for Siegel modular forms \cite{Bo}. The second aim is 
to show that a strange phenomenon occurs for low $p$-adical weights; namely, there exists a transcendental $p$-adic Eisenstein series in the quaternionic case (Theorem \ref{ThmM2}).


\section{Preliminaries}
\subsection{Notation and definitions}	
Let $\mathbb{H}$ be Hamiltonian quaternions and $\mathcal{O}$ the Hurwitz order (cf. \cite{Kri2}). The half-space of quaternions of degree $n$ is defined as 
$$
H(n;\mathbb{H}):=\{\; Z=X+iY\;|\; X,\,Y\in Her_n(\mathbb{H}),\;Y>0\;\}. 
$$
Let $J_n:=\begin{pmatrix}O_n & 1_n \\ -1_n & O_n\end{pmatrix}$. Then, the group of symplectic similitudes
$$
\left\{\;M\in M(2n,\mathbb{H})\;|\; {}^t\overline{M}J_nM=qJ_n\ {\rm for\ some\ positive\ }q\in \mathbb{R}\;\right\}
$$
acts on $H(n;\mathbb{H})$ by 
$$
Z\longmapsto M\langle Z\rangle =(AZ+B)(CZ+D)^{-1},\quad M=\begin{pmatrix}A & B \\ C & D \end{pmatrix}.
$$
Let $\varGamma_n$ denote the modular group of quaternions of degree $n$ defined by 
\begin{align*}
&G_n:=\left\{\;M\in M(2n,\mathbb{H})\;|\; {}^t\overline{M}J_nM=J_n\;\right\}, \\
&\varGamma_n:=\varGamma_n(\mathcal{O})=M(2n,\mathcal{O})\cap G_n.
\end{align*}
For a given $q\in \mathbb{N}$, the congruence subgroup 
$\varGamma _0 ^{(n)}(q)$ of $\varGamma_n$ is defined by 
$$
\varGamma _0^{(n)}(q):=\left\{ \begin{pmatrix} A & B \\ C & D \end{pmatrix}\in 
\varGamma_n \;|\; C\equiv O_n \bmod{qM(n,\mathcal{O})} \right\}.
$$
In this subsection, $\varGamma$ always denotes either $\varGamma _n$ or 
$\varGamma _0^{(n)}(q)$. 

Let $1=e_1$, $e_2$, $e_3$, $e_4$ denote the canonical basis of $\mathbb{H}$,
which is characterized by the identities
$$
e_4=e_2e_3=-e_3e_2,\quad e_2^2=e_3^2=-1.
$$
We consider the canonical isomorphism
\[
M(n,\mathbb{H}) \longrightarrow M(2n,\mathbb{C})
\]
given by $\overset{\vee}{A}=(\overset{\vee}{a}_{ij})$, where
$\overset{\vee}{a}=\begin{pmatrix}a_1+a_2i&a_3+a_4i\\-a_3+a_4i&a_1-a_2i\end{pmatrix}$,
if $a=a_1e_1+a_2e_2+a_3e_3+a_4e_4$
(cf. \cite{Kri2}).
We use the above isomorphism to define $\det(A)$ for $A\in M(n,\mathbb{H})$.
For a similitude $M=\begin{pmatrix} A & B \\ C & D \end{pmatrix}$ and
a function $f:\,H(n;\mathbb{H})\longrightarrow \mathbb{C}$,
we define the slash operator $\mid_k$ by 
$$
(f|_kM)(Z)=\det(M)^{\frac{k}{2}}\det(CZ+D)^{-k}f((AZ+B)(CZ+D)^{-1}). 
$$
A holomorphic function $f:\;H(n;\mathbb{H})\longrightarrow \mathbb{C}$ is called a
{\it quaternionic modular form of degree n and weight k for} $\varGamma $ if $f$ satisfies
$$
(f|_kM)(Z)=f(Z),
$$
for all $M\in \varGamma$. (The cusp condition is required if $n=1$.)

We denote by $M_k(\varGamma)$ the $\mathbb{C}$-vector space of all 
quaternionic modular forms of degree $n$ and weight $k$ for $\varGamma$. 
A modular form $f\in M_k(\varGamma)$ possesses a 
Fourier expansion of the form
$$
f(Z)=\sum_{0\leq H\in Her_n^\tau (\mathcal{O})}a_f(H)e^{2\pi i\tau(H,Z)},\quad
Z\in H(n;\mathbb{H}),
$$
where $Her_n^\tau (\mathcal{O})$ denotes the dual lattice of $Her_n(\mathcal{O}):=
\{S\in M(n,\mathcal{O})\,|\,{}^t\overline{S}=S\}$ with respect to
the reduced trace form $\tau$ (cf. \cite{Kri2}). For simplicity, 
we put $q^H:=e^{2\pi i\tau(H,Z)}$ 
for $H\in Her_n^\tau (\mathcal{O})$. Using this notation, 
we write the above Fourier expansion simply as $f=\sum _Ha_f(H)q^H$.

For an even integer $k$, we consider the Eisenstein series
\begin{equation}
\label{Eisen}
E_k^{(n)}(Z):=\sum_{\binom{AB}{CD}\in\varGamma_{n0}\backslash\varGamma_n}
\det(CZ+D)^{-k},\quad Z\in H(n;\mathbb{H}),
\end{equation}
where $\varGamma_{n0}:=\left\{\begin{pmatrix} A & B \\ O_n & D \end{pmatrix}\in
\varGamma_n\right\}$. It is well known that this series belongs to $M_k(\varGamma_n)$ if $k>4n-2$. 
We call this series the {\it quaternionic Eisenstein series of degree n and weight k}.
 \subsection{Fourier coefficients of Eisenstein series}
 In this section, we introduce an explicit formula for the Fourier coefficients
 of the degree 2 quaternionic Eisenstein series obtained by
 Krieg (cf. \cite{Kri3}).

Let $k>6$ be an even integer and let
$$
 E_k^{(2)}(Z)=\sum_{0\leq H\in Her_2^\tau (\mathcal{O})}a_k(H)e^{2\pi i\tau (H,Z)}
$$
be the Fourier expansion of the degree $2$ quaternionic Eisenstein series $E_k^{(2)}$. 
According to \cite{Kri3}, we introduce an explicit formula for $a_k(H)$. Given 
$O_2\ne H\in Her_2^\tau (\mathcal{O})$, the ``greatest common divisor"
of $H$ is given by
$$
\varepsilon (H):=\text{max}\{d\in\mathbb{N}\;|\; d^{-1}H\in Her_2^\tau (\mathcal{O})\;\}.
$$
\begin{Thm}[Krieg \cite{Kri3}]
\label{Four}
\it{Let $k>6$ be even and $H\ne O_2$. Then, the Fourier
coefficient $a_k(H)$ is given by:
$$
a_k(H)=\sum_{0<d|\varepsilon (H)} d^{k-1}\alpha^*(2{\rm det}(H)/d^2)
$$
and
$$
\alpha^*(\ell)=
\begin{cases}\displaystyle
-\frac{2k}{B_k} & \text{if $\ell =0$},\\
\displaystyle
-\frac{4k(k-2)}{(2^{k-2}-1)\,B_k\,B_{k-2}}[\sigma_{k-3}(\ell)-2^{k-2}\sigma_{k-3}(\ell/4)]
& \text{if $\ell\in\mathbb{N}$},
\end{cases}
$$
where $B_m$ is the $m$-th Bernoulli number and
$$
\sigma_k(m):=
\begin{cases}
0 & \text{if $m\notin \mathbb{N}$},\\
\displaystyle
\sum_{0<d|m}d^k & \text{if $m\in\mathbb{N}$}.
\end{cases}
$$
}
\end{Thm}
\subsection{${\boldsymbol U(p)}$-operator}
In the remainder of this paper, we assume that $p$ is an odd prime. For a formal power series of the form 
$F=\sum _Ha_F(H)q^H$, we define a $U(p)$ type operator as 
$$
U(p):F=\sum _Ha_F(H)q^T\longmapsto F|U(p):=\sum _Ha_F(pH)q^H. 
$$
In particular, for a modular form $F\in M_k(\varGamma  _0^{(n)}(p))$, we may regard $U(p)$ as a Hecke operator
 (cf. \cite{Bo}, \cite{Kri3}). 
We prove this in this section. More precisely, we prove that
\begin{Prop}
\label{Prop1}
\it{If $F\in M_k(\varGamma  _0^{(n)}(p))$ then $F|U(p)\in M_k(\varGamma  _0^{(n)}(p))$.}
\end{Prop}
To prove this proposition, we introduce the following lemma. 
\begin{Lem}
\label{Lem1}
\it{A complete set of representatives for the left cosets of 
$$
\varGamma  _0^{(n)}(p)\begin{pmatrix}O_n & -1_n \\ 1_n & O_n \end{pmatrix}\varGamma  _0^{(n)}(p) 
$$
is given by 
$$
\left\{ \begin{pmatrix}O_n & -1_n \\ 1_n & T \end{pmatrix}|\: T\in Her_n({\mathcal O})/pHer_n({\mathcal O}) \right\}. 
$$
}
\end{Lem}  
\begin{proof}[Proof of Lemma \ref{Lem1}] 
We set $\gamma _T:=\begin{pmatrix} O_n & -1_n \\ 1_n & T \end{pmatrix}$ and prove 
\begin{align*}
\varGamma  _0^{(n)}(p)\begin{pmatrix}O_n & -1_n \\ 1_n & O_n \end{pmatrix}\varGamma  _0^{(n)}(p) 
=\bigcup _{T\in Her_n({\mathcal O})/pHer_n({\mathcal O})}\varGamma _0^{(n)}(p) \gamma _T.  
\end{align*}
By decomposition 
\begin{equation}
\label{dec}
\begin{pmatrix}O_n & -1_n \\ 1_n & T \end{pmatrix}=\begin{pmatrix}O_n & -1_n \\ 1_n & O_n 
\end{pmatrix}\begin{pmatrix}1_n & T \\ O_n & 1_n \end{pmatrix},
\end{equation}
we easily see the inclusion 
\begin{equation}
\label{inc}
\varGamma  _0^{(n)}(p)\begin{pmatrix}O_n & -1_n \\ 1_n & O_n \end{pmatrix}\varGamma  _0^{(n)}(p) \supset 
\bigcup _{T\in Her_n({\mathcal O})/pHer_n({\mathcal O})}\varGamma  _0^{(n)}(p) \gamma _T.  
\end{equation}
We shall prove the converse inclusion. Note that $T\equiv T'$ mod $pHer_n({\mathcal O})$ if and only if 
$\varGamma  _0^{(n)}(p)\gamma _T=\varGamma  _0^{(n)}(p)\gamma _{T'}$. Hence, we have 
$$
\bigcup _{T\in Her_n({\mathcal O})/pHer_n({\mathcal O})}\varGamma  _0^{(n)}(p) \gamma _T=\bigcup _{T\in Her_n({\mathcal O})}\varGamma  _0^{(n)}(p)\gamma _T
$$
as a set. Again, by the decomposition (\ref{dec}), it suffices to show that, for any 
$\begin{pmatrix}A & B \\ C & D\end{pmatrix}\in \varGamma  _0^{(n)}(p)$, there exists $S\in Her_n({\mathcal O})$ such that 
$$
\begin{pmatrix}O_n & -1_n \\ 1_n & T \end{pmatrix}\begin{pmatrix}A & B \\ C & D \end{pmatrix}
\begin{pmatrix}O_n & -1_n \\ 1_n & S\end{pmatrix}^{-1}\in \varGamma  _0^{(n)}(p). 
$$
A direct calculation shows that
\begin{align*}
&\begin{pmatrix}O_n & -1_n \\ 1_n & T \end{pmatrix}\begin{pmatrix}A & B \\ C & D \end{pmatrix}\begin{pmatrix}O_n & -1_n \\ 1_n & S
 \end{pmatrix}^{-1}=\begin{pmatrix}-CS+D & -C \\ (A+TC)S-(B+TD) & A+TC \end{pmatrix}. 
\end{align*}
Hence, the proof is reduced to finding $S\in Her_n({\mathcal O})$ such that $AS\equiv B+TD$ mod 
$p\,M(n,{\mathcal O})$. 
Recall that $A{}^t\overline{D}-B{}^t\overline{C}=1_n$ and hence $A{}^t\overline{D}\equiv 1_n$ mod $p\,M(n,{\mathcal O})$. 
If we choose $S$ as $S:={}^t\overline{D}(B+TD)$, then $AS\equiv B+TD$ mod $p\,M(n,{\mathcal O})$. 
To complete the proof, we need to show that $S={}^t\overline{D}(B+TD)\in Her_n({\mathcal O})$. This assertion comes from the fact that 
${}^t\overline{D}B$, ${}^t\overline{D}TD\in Her _n({\mathcal O})$. 
\end{proof}
We now return to the proof of Proposition \ref{Prop1}. 
\begin{proof}[Proof of Proposition \ref{Prop1}] 
Let $F\in M_k(\varGamma  _0^{(n)}(p))$. From Lemma \ref{Lem1}, we have 
\begin{align*}
F|\varGamma  _0^{(n)}(p)\begin{pmatrix}O_n & -1_n \\ 1_n & O_n \end{pmatrix}\varGamma  _0^{(n)}(p)&=
\sum_TF|_k \begin{pmatrix}O_n & -1_n \\ 1_n & T \end{pmatrix}\\
&=\sum _TF|W_p|_k\begin{pmatrix}1_n & T \\ O_n & p1_n \end{pmatrix},
\end{align*}
where $W_p$ is the Fricke involution
$$
F\longmapsto F|W_p:=F|_k\begin{pmatrix}O_n & -1_n \\ p1_n & O_n\end{pmatrix}.
$$
We see by the usual way that $F|W_p\in M_k(\varGamma _0^{(n)}(p))$. If we write $G=F|W_p=\sum _Ha_G(H)q^H$, then
\begin{align*}
\sum_TF|W_p|_k\begin{pmatrix}1_n & T \\ O_n & p1_n \end{pmatrix}&=\sum _TG|_k\begin{pmatrix}1_n & T \\ O_n & p1_n \end{pmatrix}\\ 
&=\sum _H\left(\sum _Te^{\frac{2\pi i}{p}\tau(H,T)}\right)a_G(H)e^{\frac{2\pi i}{p}\tau(H,Z)}\\
&=c\cdot G|U(p), 
\end{align*}
where $c:=\sharp Her _n({\mathcal O})/pHer _n({\mathcal O})$ and the last equality follows from the following lemma.
\begin{Lem}
\label{Lem2}
\it{
For fixed $H\in Her _n^\tau ({\mathcal O})$, we have
\begin{align}
\label{sum} 
\sum _Te^{\frac{2\pi i}{p}\tau (H,T)}=\begin{cases} 0 \quad &{\rm if}\quad H\not \in pHer^\tau _n({\mathcal O}), \\ c \quad &{\rm if}
\quad H\in pHer^\tau _n({\mathcal O}). \\ \end{cases}
\end{align}
}
\end{Lem}
\begin{proof}[Proof of Lemma \ref{Lem2}] 
For $H\in Her_n^\tau ({\mathcal O})$, we define
\begin{align*}
G(H):=\sum _{T\in Her_n({\mathcal O})/pHer_n({\mathcal O})}e^{\frac{2\pi i}{p}\tau (H,T)}. 
\end{align*}
This definition is independent of the choice of the representation $T$. Replacing $T$ by $T+S$, we obtain 
$$
G(H)=G(H)e^{\frac{2\pi i}{p}\tau (H,S)}.
$$
Hence, $G(H)=0$ unless $e^{\frac{2\pi i }{p}\tau (H,S)}=1$; i.e., $\tau (H,S)\in p\mathbb{Z}$. This implies 
$\tau (\frac{1}{p}H,S)\in \mathbb{Z}$ 
for all $S\in Her_n ({\mathcal O})$. The definition of a dual lattice yields 
$$
\frac{1}{p}H\in Her_n^\tau ({\mathcal O}). 
$$
\end{proof}
From this lemma, we have
$$
F|\varGamma  _0^{(n)}(p)\begin{pmatrix}O_n & -1_n \\ 1_n & O_n \end{pmatrix}\varGamma  _0^{(n)}(p)=c\cdot F|W_p|U(p). 
$$
Hence, the action of $U(p)$ is described by the action of the double coset 
$$
\varGamma  _0^{(n)}(p)\begin{pmatrix}1_n & O_n \\ O_n & p1_n\end{pmatrix}\varGamma  _0^{(n)}(p).
$$
Therefore, we have $F|U(p)\in M_k(\varGamma  _0^{(n)}(p))$, which completes the proof of
Proposition \ref{Prop1}.
\end{proof} 
\begin{Rem}
The proof of Lemma \ref{Lem2} is due to Krieg.
\end{Rem}
\section{Main results}
\subsection{Modularity of $p$-adic Eisenstein series}
In this subsection, we deal with a suitable constant multiple of 
the normalized quaternionic Eisenstein series 
$$
G_k=G^{(2)}_k:=(2^{k-2}-1)\frac{B_kB_{k-2}}{4k(k-1)}E_k^{(2)}
$$
and show that certain $p$-adic limits of this Eisenstein series are ``real" modular forms for $\varGamma  _0^{(2)}(p)$. 

We write the Fourier expansion of $G_k$ as $G_k=\sum _{H}b_k(H)q^H$. We remark that
\[b_k(O_2)=(2^{k-2}-1)\frac{-B_kB_{k-2}}{4k(k-2)}.\]

For an odd prime $p$ we put 
\begin{align*}
G^*_k&:=\frac{-1}{1+p^{k-3}}\left\{p^{2(k-3)}(G_k|U(p)-p^{k-1}G_k)-(G_k|U(p)-p^{k-1}G_k)|U(p)\right\}\\
&\in M_k(\varGamma  _0^{(n)}(p)), 
\end{align*}
where this modularity follows from Proposition \ref{Prop1}. The first main theorem is
\begin{Thm}
\label{ThmM}
\it{Let $p$ be an odd prime and $k$ an even integer with $k\ge 4$. Define a sequence 
$\{k_m\}$ by 
$$k_m:=k+(p-1)p^{m-1}. $$
Then, the corresponding sequence of Eisenstein series $\{G_{k_m}\}$ has a $p$-adic limit 
$G^*_k$ and we have
\begin{align}
\label{Lim}
\lim _{m\to \infty }G_{k_m}=G^*_k\in M_k(\varGamma  _0^{(2)}(p)).
\end{align}
}
\end{Thm}

\begin{proof}
The proof of (\ref{Lim}) is reduced to show that $G_k^*$ is obtained by
removing all $p$-factors of the Fourier coefficients of the quaternionic Eisenstein series.

To calculate the Fourier coefficients of $G^*_k$, we set 
$$
F_k=G_k|U(p)-p^{k-1}G_k. 
$$
We can then rewrite $G^*_k$ as 
$$
G^*_k=\frac{-1}{1+p^{k-3}}(p^{2(k-3)}F_k-F_k|U(p)). 
$$
We write the Fourier expansions as 
$$
G^*_k=\sum _HA_k(H)q^H,\quad F_k=\sum _HB_k(H)q^H.  
$$
First, we calculate the constant term of $G^*_k$. 
Since
$$
b_k(O_2)=(2^{k-2}-1)\frac{-B_kB_{k-2}}{4k(k-2)},
$$
the constant term of $G^*_k$ becomes
\begin{align*}
A_k(O_2)&=\frac{-1}{1+p^{k-3}}\{p^{2(k-3)}(b_k(O_2)-p^{k-1}b_k(O_2))-(b_k(O_2)-p^{k-1}b_k(O_2))\}\\
&=(1-p^{k-1})(1-p^{k-3})(2^{k-2}-1)\frac{-B_kB_{k-2}}{4k(k-2)}. 
\end{align*}

Second, we calculate the coefficient $A_k(H)$ for $H$ with ${\rm rank}(H)=1$. 
\begin{align*}
B_k(H)&=b_k(pH)-p^{k-1}b_k(H)\\
&=(2^{k-2}-1)\frac{B_{k-2}}{2(k-2)}\left(\sum_{0<d|p\varepsilon (H)}d^{k-1}-p^{k-1}\sum_{0<d|\varepsilon (H) }d^{k-1}\right)\\
&=(2^{k-2}-1)\frac{B_{k-2}}{2(k-2)}\sigma _{k-1}^*(\varepsilon (H)),
\end{align*}
where $\sigma _{m}^*(N)$ is defined as 
$$
\sigma _{m}^*(N):=\sum _{\substack{0<d|N \\ (p,d)=1}}d^{m}. 
$$
Note that $B_k(pH)=B_k(H)$ when ${\rm rank}(H)=1$. Hence, we have 
$$
A_k(H)=(1-p^{k-3})(2^{k-2}-1)\frac{B_{k-2}}{2(k-2)}\sigma _{k-1}^*(\varepsilon (H)).  
$$

Finally, we consider the case ${\rm rank}(H)=2$.  
\begin{align*}
B_k(H)&=b_k(pH)-p^{k-1}b_k(H)\\
&=\sum _{0<d| p\varepsilon (H)}d^{k-1}[\sigma _{k-3}\left(\tfrac{2p^2\det H}{d^2}\right)
-2^{k-2}\sigma _{k-3}\left(\tfrac{2p^2\det H}{4d^2}\right)]\\
&-p^{k-1}\sum _{0<d| \varepsilon (H)}d^{k-1}[\sigma _{k-3}
\left(\tfrac{2\det H}{d^2}\right)-2^{k-2}\sigma _{k-3}\left(\tfrac{2\det H}{4d^2}\right)]\\
&=\sum _{\substack{ 0<d|\varepsilon (H) \\ (p,d)=1}}d^{k-1}[\sigma _{k-3}
\left(\tfrac{2p^2\det H}{d^2}\right)-2^{k-2}\sigma _{k-3}\left(\tfrac{2p^2\det H}{4d^2}\right)].
\end{align*}
Here, the last equality was obtained from the elemental property: 
\begin{Lem}
\label{Lem0}
\it{
Let $p$ be a prime and $N$ a positive integer. For a function $f:\mathbb{N}\rightarrow \mathbb{N}$, the following holds:
\begin{align*}
\sum _{0<d|pN}f(d)=\sum _{\substack{ 0<d|N \\ (p,d)=1}}f(d)+\sum _{0<d|N}f(pd). 
\end{align*}
} 
\end{Lem}
Therefore, 
\begin{align*}
A_k(H)&=\frac{-1}{1+p^{k-3}}(p^{2(k-3)}B_k(pH)-B_k(H))\\
&=\frac{-1}{1+p^{k-3}}\Big( p^{2(k-3)}\sum _{\substack{ 0<d|\varepsilon (H) \\ (p,d)=1}}d^{k-1}
[\sigma _{k-3}\left(\tfrac{2p^2\det H}{d^2}\right)-2^{k-2}\sigma _{k-3}\left(\tfrac{2p^2\det H}{4d^2}\right)] \\ 
&-\sum _{\substack{ 0<d|\varepsilon (H) \\ (p,d)=1}}d^{k-1}[\sigma _{k-3}\left(\tfrac{2p^4\det H}{d^2}\right)-2^{k-2}
\sigma _{k-3}\left(\tfrac{2p^4\det H}{4d^2}\right)] \Big).
\end{align*}
By repeatedly applying Lemma \ref{Lem0}, we obtain
$$
p^{2m}\sigma _{m}(N)-\sigma _{m}(p^2N)=-(1+p^{m})\sum _{\substack{ 0<d|N \\ (p,d)=1}}d^m. 
$$
From this, we have 
$$
A_k(H)=\sum _{\substack{ 0<d|\varepsilon (H) \\ (p,d)=1}}d^{k-1}[\sigma ^*_{k-3}\left(\tfrac{2\det H}{d^2}\right)-2^{k-2}\sigma ^*_{k-3}
\left(\tfrac{2\det H}{4d^2}\right)].
$$
Summarizing these calculations, we obtain the following formula: 
\begin{Prop}
\it{The following holds:} 
\begin{align*}
A_k(H)=\begin{cases}
\displaystyle (1-p^{k-1})(1-p^{k-3})(2^{k-2}-1)\frac{-B_kB_{k-2}}{4k(k-2)},\ & \text{if}\; H=O_2,
\vspace{2mm}\\
\displaystyle (1-p^{k-3})(2^{k-2}-1)\frac{B_{k-2}}{2(k-2)}\sigma _{k-1}^*(\varepsilon (H)),\ &
\text{if\;rank}(H)=1,
\vspace{2mm}\\ 
\displaystyle \sum _{\substack{ 0<d|\varepsilon (H) \\ (p,d)=1}}d^{k-1}[\sigma ^*_{k-3}\left(\tfrac{2\det H}{d^2}\right)-2^{k-2}\sigma ^*_{k-3}
\left(\tfrac{2\det H}{4d^2}\right)],\ &
\text{if\;rank}(H)=2. 
\end{cases}
\end{align*}
\end{Prop} 
On the other hand,
\begin{align*}
b_{k_m}(H)=
\begin{cases}
\displaystyle (2^{k_m-2}-1)\frac{-B_{k_m}B_{k_m-2}}{4k_m(k_m-2)},\ &
\text{if}\; H=O_2,
\vspace{2mm}\\
\displaystyle (2^{k_m-2}-1)\frac{B_{k_m-2}}{2(k_m-2)}\sigma _{k_m-1}(\varepsilon (H)),\ &
\text{if\;rank}(H)=1,
\vspace{2mm}\\ 
\displaystyle \sum _{0<d|\varepsilon (H)}d^{k_m-1}[\sigma_{k_m-3}
\left(\tfrac{2\det H}{d^2}\right)-2^{k_m-2}\sigma_{k_m-3}\left(\tfrac{2\det H}{4d^2}\right)],\ &
\text{if\;rank}(H)=2. 
\end{cases}
\end{align*}
Combining these formulas and the Kummer congruence, we can prove that
$$	
\lim _{m\to \infty }b_{k_m}(H)=A_k(H)
$$
for all $H\in Her_2^\tau(\mathcal{O})$. This completes the proof of Theorem \ref{ThmM}. 
\end{proof} 
\begin{Rem}
\it{ 
Following Hida \cite{Hida}, our $G_k^*$ can be p-adic analytically interpolated with respect to the weight.
}
\end{Rem}

\subsection{Transcendental $p$-adic Eisenstein series}

As we have seen in the previous section, under certain conditions, a $p$-adic limit of a quaternionic 
Eisenstein series becomes a ``real" modular form
with rational Fourier coefficients. This also holds for Siegel Eisenstein and Hermitian 
Eisenstein series. More precisely, they coincide with the genus theta series 
(cf. \cite{KN}, \cite{N}). In these cases (Siegel, Hermitian cases), the $p$-adic Eisenstein 
series is algebraic. We shall show that there exists an example of a transcendental 
$p$-adic Eisenstein series for quaternionic modular forms. 

The second main theorem is
\begin{Thm}
\label{ThmM2}
\it{
Let $p$ be an odd prime and $\{k_m\}$ the sequence defined by 
$$k_m:=2+(p-1)p^{m-1}.$$
Then, the $p$-adic Eisenstein series $\displaystyle\widetilde{E}=\lim_{m\to\infty}E_{k_m}^{(2)}$ 
is transcendental;
namely, $\widetilde{E}$ has transcendental coefficients
where $E_k^{(2)}$ is the normalized quaternionic Eisenstein series of degree 2 defined 
in (\ref{Eisen}).
}
\end{Thm}
\begin{proof}
We calculate $\tilde{a}(H):=\displaystyle\lim_{m\to\infty}a_{k_m}(H)$ at 
$H=\begin{pmatrix}1&\tfrac{e_1+e_2}{2}\\\tfrac{e_1-e_2}{2}&1\end{pmatrix}\in Her_2^\tau (\mathcal{O})$.
The convergence for general $H$ is proved similarly. 

It follows from Theorem \ref{Four} that
$$
a_{k_m}(H)=-\frac{4k_m(k_m-2)}{(2^{k_m-2}-1)B_{k_m}B_{k_m-2}}.
$$
(We note that $\varepsilon (H)=1$ and $\det (H)=\frac{1}{2}$.) 
We rewrite the right-hand side as
$$
-4\cdot \frac{2+(p-1)p^{m-1}}{B_{2+(p-1)p^{m-1}}}\cdot
\frac{1}{B_{(p-1)p^{m-1}}}\cdot
\frac{p^m}{2^{(p-1)p^{m-1}}-1}\cdot
\frac{p-1}{p}
$$
and calculate the $p$-adic limit separately:\\
(i)\; $\displaystyle\lim_{m\to\infty}\frac{2+(p-1)p^{m-1}}{B_{2+(p-1)p^{m-1}}}=\frac{B_2}{2}=\frac{1}{12}$.\\
This is a consequence of the Kummer congruence.\\
(ii)\; $\displaystyle\lim_{m\to\infty}B_{(p-1)p^{m-1}}=\frac{p-1}{p}$.\\
This identity comes from the fact that the residue of the $p$-adic $L$-function $L_p(s,\chi^0)$ at
$s=0$ is just $1-\frac{1}{p}$.\\
(iii)\; $\displaystyle\lim_{m\to\infty}\frac{2^{(p-1)p^{m-1}}-1}{p^m}=\frac{\log_p(2^{p-1})}{p}$.\\
where $\log_p$ is the $p$-adic logarithmic function defined by
$$
\log_p(x)=x-\frac{x^2}{2}+\frac{x^3}{3}-\cdots ,\qquad (|x|_p<1).
$$
Leopoldt's formula \cite{L} states that
$$
\lim_{m\to\infty}\frac{x^{p^m}-1}{p^m}=\log_p(x).
$$
if $|x-1|_p<1$ .
This implies that
$$
\lim_{m\to\infty}\frac{2^{(p-1)p^{m-1}}-1}{p^m}=\frac{1}{p}\cdot\log_p(2^{p-1}).
$$
Combining these formulas, we obtain
\begin{equation}
\tilde{a}(H)=\lim_{m\to\infty}a_{k_m}(H)=\frac{-48p}{\log_p(2^{p-1})}.
\end{equation}
We shall show that $\log_p(2^{p-1})$ is transcendental. Let $\text{exp}_p$ be the $p$-adic exponential function defined by
$$
\text{exp}_p(x)=1+x+\frac{x^2}{2!}+\frac{x^3}{3!}+\cdots,\qquad (|x|_p<p^{-\frac{1}{p-1}}).
$$
It is known that if $|x|_p<p^{-\frac{1}{p-1}}$, then
$$
\text{exp}_p(\log_p(1+x))=1+x,\quad (\text{e.g., \cite{G}}).
$$
To prove the transcendency of $\log_p(2^{p-1})$, we use the following theorem by Mahler:
\begin{Thm}[Mahler \cite{M}]\it{
Let $\mathbb{C}_p$ be the completion of the algebraic closure of $\mathbb{Q}_p$.
For any algebraic over $\mathbb{Q}$ $p$-adic number
$\alpha\in\mathbb{C}_p$ with
$0<|\alpha|_p<p^{-\frac{1}{p-1}}$, the quantity
$\text{exp}_p(\alpha)$ is transcendental.}
\end{Thm}
We note that $|x|_p<p^{-\frac{1}{p-1}}$ is equivalent to $|x|_p<1$ for odd prime $p$
(e.g., \cite{G}, p.114). We put $\alpha=2^{p-1}-1$. Since $|\alpha|_p<1$,
we have
$$
\text{exp}_p(\log_p(1+\alpha))=1+\alpha=2^{p-1}.
$$
The right-hand side is obviously algebraic. Hence, by Mahler's theorem,
$\log_p(1+\alpha)=\log_p(2^{p-1})$ must be transcendental. Thus, we can
prove the transcendency of $\tilde{a}(H)$ at
$H=\begin{pmatrix}1&\frac{e_1+e_2}{2}\\ \frac{e_1-e_2}{2}&1\end{pmatrix}$.
This completes the proof of Theorem \ref{ThmM2}. 
\end{proof}

\begin{Rem}
By the above proof, we see that all coefficients $\tilde{a}(H)$ corresponding to
$H$ with rank $2$ are transcendental. However, $\tilde{a}(H)$ for $H$ with
$\text{rank}(H)\leq 1$ are rational.
\end{Rem}

\noindent
\textbf{Acknowledgments:}
We would like to thank Professor A. Krieg for helpful comments on the proof of
the modularity of $f|U(p)$. 
We also thank Professor M.~Amou for pointing out the transcendency of 
$\log_p(2^{p-1})$.

Toshiyuki {\scshape Kikuta}\\
Department of Mathematics\\
Interdisciplinary Graduate School of\\
Science and Engineering
Kinki University\\
Higashi-Osaka 577-8502, Japan\\
E-mail:\; kikuta84@gmail.com\\
\\
\\
Shoyu {\scshape Nagaoka}\\
Department of Mathematics\\
School of Science and Engineering\\
Kinki University\\
Higashi-Osaka 577-8502, Japan\\
E-mail:\; nagaoka@math.kindai.ac.jp

\end{document}